\documentclass[reqno,final]{amsart}
\usepackage{natbib}  %Nature-like bibliography
\usepackage{fancyhdr} %Headers
\usepackage{color} %Color definition
\usepackage{hyperref} %Internal and external links
\usepackage{graphicx} %Graphics inclusion

\usepackage{pstricks}
\usepackage{amssymb}

\definecolor{aleacolor}{rgb}{0.16,0.59,0.78}

\hypersetup{
breaklinks,
colorlinks=true,
linkcolor=aleacolor,
urlcolor=aleacolor,
citecolor=aleacolor}

\renewcommand{\cite}{\citet}

\theoremstyle{plain}
\newtheorem{theorem}{Theorem}[section]                                          
\newtheorem{proposition}[theorem]{Proposition}                          

\newtheorem{corollary}[theorem]{Corollary}
\newtheorem{conjecture}[theorem]{Conjecture}
\theoremstyle{definition}

\theoremstyle{remark}
\newtheorem{remark}[theorem]{Remark}

\makeatletter \@addtoreset{equation}{section} \makeatother

\usepackage{nicefrac}

\def\Brown{B}
\newcommand{\Evc}[2]{\E \left ( #1 \middle | #2 \right )}

\newcommand{\abs}[1]{\left\vert#1\right\vert}

\newcommand{\setb}[2]{\left \{ #1 \ \middle | \ #2 \right \} }

\newcommand{\wh}[1]{\widehat{#1}}

\def\E{\mathbb E}

\def\e{\mathrm e}
\def\im{\mathrm i}
\def\di{\mathrm d}

\def\1{{\mathsf 1}}
\def\tem{\textemdash}

\def\Pr{\operatorname{Prob}} %%% Probability
\def\diam{\operatorname{diam}} %%% diameter
\def\dist{\operatorname{dist}}   %%%  distance
\def\ra{\rightarrow}

\begin{document}

\title{Trapping planar Brownian motion in a non circular trap}
\author{Jeffrey Schenker}
\address{Department of Mathematics \newline Michigan State University \newline 619 Red Cedar Road\newline East Lansing MI 48824}

\email{jeffrey@math.msu.edu}
\urladdr{\url{http://math.msu.edu/~jeffrey}}
\thanks{Supported by NSF Award DMS-1411411 {\it Interpreting Data from Trapping of Stochastic Movers}.}

\begin{abstract}
	Brownian motion in the plane in the presence of a ``trap" at which motion is stopped is studied .  If the trap $T$ is a connected compact set, it is shown that the probability for planar Brownian motion to hit this set before a given time $t$ is well approximated even at short times by the probability that Brownian motion  hits a disk of radius $r_T$ equal to the conformal radius of the trap $T$.\end{abstract}
\maketitle

\section{Introduction} The ability to compute hitting probabilities for random movers like Brownian motion in the plane  is a problem of intrinsic interest and of importance in a number of applications.  In particular, it is fundamental to current work on the theory of trapping of randomly moving organisms such as small insects \cite{Miller2015,TrappingPaper}, with applications to chemical ecology and agricultural pest management.

Consider a connected, compact set $T$ in the plane, that is neither empty nor a single point.
Let $P_T(z,t)$ denote the probability that a standard planar Brownian motion started at $z$ hits $T$ by time $t$.  If $T$ is a disk, there is an exact formula for $P_T$ as an integral of Bessel functions \cite{Wendel1980}; see eq.\ \eqref{eq:integralformula} below \tem \ this identity goes back to work of \cite{Nicholson1921} on heat flow.
For other compact sets no exact formula is known, although the following asymptotic 
\begin{equation}\label{eq:Hunt}\lim_{t \rightarrow \infty} \ln t \left ( 1-P_T(z,t) \right ) \ = \ 2 \pi H_T(z) 	
\end{equation}
holds in general. Here $H_T(z)$ is the \emph{Green's function for $T$ with pole at $\infty$}, which is the unique harmonic function in the exterior of $T$ that vanishes as $z\rightarrow T$ and has the asymptotic growth $H_T(z) \sim \frac{1}{\pi} \ln |z|$ as $z\rightarrow \infty$.  Eq.\ \eqref{eq:Hunt} was conjectured by Kac and proved by \cite{Hunt1956}. See also \cite{Collet2000}, in which asymptotics for the sub-probability distribution of the movers that have avoided the trap are derived in 2 and higher dimensions. Hunt  gave no explicit error bound, but we will see below that his argument is easily extended to show that the error made in the approximation 
\begin{equation}\label{eq:Huntapprox} p_T(z,t) \ \approx \ 1 - \frac{2\pi H_T(z)}{\ln t}
\end{equation}
is of order $H_T(z) \nicefrac{\ln \ln^2 t}{\ln^2 t}$. 

In the ecological applications mentioned above, it has been important to compute probabilities over intermediate time scales, which are not in the regime covered by the asymptotic result \eqref{eq:Hunt}.  In those applications $T$ is a ``trap'' at which the movers are captured.  Typically the trap consists of two parts: a physical trap which ensnares the organisms that contact it and a chemical plume of an attractant, such as a pheromone, emitted by the physical trap.  In general the plume is of unknown shape and size, but upon entering the plume an organism ceases to engage in random search and instead flies upwind to the physical trap at which it is captured. Thus the difficulty in computing $P_T$ is compounded by the fact that we do not even know the precise set $T$! 

One could also worry that the plume may be dynamic, and indeed this is surely the case in general.  Although the molecular weight of the attractant is often quite large, causing the plume to have very low diffusivity, passive transport of the plume by shifts in wind direction is probably relevant in the field. However, it seems likely that such dynamical behavior of the plume will introduce an additional averaging that will only enhance the phenomenon described below.  Thus for the present work we restrict our attention to static plumes.  

Progress is possible on this problem because we are primarily interested in the response of a uniformly distributed population to the trap, for which it suffices to know circular averages of $P_T$:
$$p_T(r,t) \ := \ \frac{1}{2\pi r} \int_{|z|=r} P_T(z,t) |\di z|,$$
which give the hitting probability for a mover with initial condition randomly distributed on a circle of radius $r$.  Here $|\di z|$ denotes arc-length measure on the circle $|z|=r$.
In applications, we can measure $p_T(r,t)$ by releasing individuals uniformly on a circle of a given radius \cite{TrappingPaper}. Averaging the approximation eq.\ \eqref{eq:Huntapprox} we obtain
\begin{equation}\label{eq:avgHuntapprox}
	p_T(r,t) \ \approx \ 1 - \frac{1}{2\pi r} \int_{|z|=r} \frac{2\pi H_T(z)}{\ln t} |\di z| \ = \ 1 - \frac{2 \ln \nicefrac{r}{r_T}}{\ln t} ,
\end{equation}
where $r_T$ is the so-called \emph{conformal radius}, or \emph{logaritheoremic capacity}, of $T$, defined by the relation\footnote{The equality of eq.\ \eqref{eq:avgHuntapprox} follows from the mean value property for the harmonic function $H_T(z)-\ln |z|$ defined in a neighborhood of $z=\infty$.}
$$\ln r_T \ := \ \lim_{z \rightarrow \infty} \pi H_T(z) - \ln |z|.$$
The right hand side of eq.\ \eqref{eq:avgHuntapprox} is the same for any trap $T$ with conformal radius $r_T$.  In particular, eq.\ \eqref{eq:avgHuntapprox} implies that, asymptotically, $p_T$ agrees with the hitting probability for a disk of radius $r_T$. The aim of this paper is to demonstrate that the approximation is greatly improved and becomes valid at much shorter time scales if we use the exact probability to hit a disk of radius $r_T$ in place of $1 - \nicefrac{2\ln \nicefrac{r}{r_T}}{\ln t}.$

 Let $p_{r_T}^D(r,t)$ denote the probability that a standard planar Brownian motion released at a distance $r$ from the origin hits a disk of radius $r_T$ centered at the origin before time $t$.  Abelian averages of $p_{r_T}^D$ are given by the formula \cite{Wendel1980}:
\begin{equation}\label{eq:diskf}f_{r_T}^D(r,\tau) \ = \ \frac{1}{\tau} \int_0^\infty \e^{-\nicefrac{t}{\tau}} p_{r_T}^D(r,t) \di t \ = \ \frac{K_0(\sqrt{\nicefrac{2 r^2}{\tau}})}{K_0(\sqrt{\nicefrac{2r_T^2}{\tau}})}	
\end{equation}
where $K_0$ is the order zero modified Bessel function of the second kind.\footnote{Eq.\ \eqref{eq:diskf} follows from the fact that $f_{r_T}^D$ solves the boundary value problem $$\frac{2}{\tau} f_{r_T}^D(r,\tau) = \partial_r^2 f_{r_T}^D(r,\tau) + \frac{1}{r}\partial_r f_{r_T}^D(r,\tau)$$
with $f_{r_T}^D(r_T,\tau)=1$ and $f_{r_T}^D(r,\tau) \ra 0$ as $r\ra \infty$. See \S\ref{sec:proofs} bellow and also \cite{Nicholson1921,Carslaw1940a}.}
  The main new result presented here is the following
\begin{theorem}\label{theorem:new} Let $f_T(r,\tau)$ denote the Abelian mean of $p_T(r,t)$,
$$ f_T(r,\tau) \ = \ \frac{1}{\tau}\int_0^\infty \e^{\nicefrac{t}{\tau}} p_T(r,t) \di t.$$
Let $r_T$ be the conformal radius of $T$ and let 
$$d = \max\left ( \diam (T),\e^{\gamma} r_T\right ) $$
where $\gamma = 0.5772\cdots$ is the Euler-Mascheroni constant.
If $\tau > \frac{\e}{2} d^2$  and $|z| \ge r_0 := \max_{w\in T} |w|$ then
\begin{equation}\label{eq:newestimate}\abs{f_T(r,\tau) - f_{r_T}^D(r,\tau)} \ \le \ 2.9 \, \frac{d^2}{\tau} f_{r_T}^D(r,\tau). 	
\end{equation}
\end{theorem}
The main advantages of the bound provided by Theorem \ref{theorem:new} are twofold.  First, it estimates the \emph{relative error} of approximating $f_T(r,t)$ by $f_{r_T}^D(r,\tau)$.  Second, the estimate provided is uniform for all $r\ge r_0$. Thus for $\tau \ge 30 d^2$, say, we know that $f_{r_T}^D(r,\tau)$ estimates $f_T(r,\tau)$ to within one part in ten \emph{for every $r$}.

By way of contrast, in the following Theorem, which is the basis for eq.\ \eqref{eq:Hunt}, the error is absolute and not uniform in $|z|$. 
\begin{theorem}\label{theorem:Hunt}
Let $F_T(z,\tau)$ denote the Abelian mean of $P_T(z,t)$,
$$ F_T(z,\tau) \ = \ \frac{1}{\tau}\int_0^\infty \e^{\nicefrac{t}{\tau}} P_T(z,\tau) \di t.$$
Let $r_T$ be the conformal radius of $T$, $d$ be as in Theorem \ref{theorem:new},  $$\tau_0 = \frac{1}{2} \e^{2\gamma} r_T^2 $$
and 
$$R_z = \max\left(\sup_{w\in T} |w-z|, \e^{\gamma} r_T \right ). $$
If $\tau > \frac{\e}{2} d^2 $ then
\begin{equation} F_T(z,\tau) \ \ge \ 1 \ - \ \frac{2 \pi H_T(z)}{\ln \nicefrac{\tau}{\tau_0} } \ - \ 0.8 \ \frac{d^2}{ \tau}    .\label{eq:Huntlower}	
\end{equation}
If $\tau > \frac{\e}{2} R_z^2$  then
\begin{equation}\label{eq:Huntupper} F_T(z,\tau) \ \le \ 1 \ - \ \frac{2 \pi H_T(z)}{\ln \nicefrac{\tau}{\tau}_0  } \ + \ 0.8 \ \frac{R_z^2}{\tau} .	
\end{equation}
\end{theorem}
\begin{remark} As we show below, Hunt's asymptotic eq.\ \eqref{eq:Hunt} follows from these two estimates.  However, Theorem \ref{theorem:Hunt} goes beyond the result in \cite{Hunt1956} in that we extract explicit error bounds from the proof.
\end{remark}

Since $F_T(z,\tau) \le 1$ in any case, we see that eq.\ \eqref{eq:Huntupper} is trivial unless
$$ - \ \frac{2 \pi H_T(z)}{\ln \nicefrac{\tau}{\tau}_0  } \ + \ 0.8 \ \frac{R_z^2}{\tau} < 0,$$
i.e., unless
$$ 0.8 \frac{R_z^2}{2 \pi H_T(z)} < \frac{\tau}{\ln \nicefrac{\tau}{\tau_0}}.$$
For large $z$, $R_z^2 \sim |z|^2$ and $\pi H_T(z) \sim \ln \nicefrac{|z|}{r_T}$.  Thus this amounts to the condition
$$ 0.8 \frac{|z|^2}{2 \ln \nicefrac{|z|}{r_T}} \ < \ \frac{\tau}{\ln \nicefrac{\tau}{\tau_0}},$$
which is roughly $\tau > 0.4 |z|^2$.
On the other hand the estimates afforded by Theorem \ref{theorem:new} are valid for all $r$ once $\tau>\nicefrac{\e d^2}{2}$.  

Theorems \ref{theorem:new} and \ref{theorem:Hunt}  present  estimates  for the Abelian averages of hitting probabilities.  In the case of Theorem \ref{theorem:Hunt}, since $P_T(z,t)$ is an increasing function of $t$, it is straightforward to derive estimates that hold point-wise in $t$.  Indeed we have the following
\begin{corollary}\label{cor:Hunt}
For fixed $z$ in the exterior of $T$ we have
$$\abs{P_T(z,t) - \frac{2\pi H_T(z)}{\ln \nicefrac{t}{\tau_0}}} \ \le \ \frac{\ln \ln^2\nicefrac{t}{\tau_0}}{\ln^2\nicefrac{t}{\tau_0}} \left [ 2 \pi H_T(z) + o(1) \right ]$$
as $t \rightarrow \infty.$  In particular, the error in the approximation eq.\ \eqref{eq:Huntapprox} is essentially of order $ H_T(z) \frac{\ln \ln^2\nicefrac{t}{\tau_0}}{\ln^2\nicefrac{t}{\tau_0}}$.
\end{corollary}
On the other hand, because 
$$ \Delta_T(r,t) \ := \ \abs{p_T(r,t) - p_{r_T}^D(r,t)},$$
is not obviously monotone in $t$, it is not clear how to transform the estimate eq.\ \eqref{eq:newestimate} provided by Theorem \ref{theorem:new} into a point-wise bound for $\Delta_T(r,t)$.  Nonetheless, we demonstrate by numerical computation that, for $T$ a line segment, $\Delta_T(r,t)$ appears to converge \emph{very rapidly} to zero, motivating the following
\begin{conjecture}
	For a given compact planar region $T$ there is a decreasing function $E_T(t)$ of time $t$ that goes to zero rapidly as $t \rightarrow \infty$ such that
	\begin{equation}\label{eq:conj}
		\abs{p_T(r,t) - p_{r_T}^D(r,t)} \ \le \ E_T(t) p_{r_T}^D(r,t)
	\end{equation}
	and
	\begin{equation}\label{eq:conj2}
	\abs{p_T(r,t) - p_{r_T}^D(r,t)} \ \le \ E_T(t) (1-p_{r_T}^D(r,t))
	\end{equation}
	for all $r \ge r_0=\max_{w\in T} |w|$, where $r_T$ is the conformal radius of $T$.\label{conj:conj}
\end{conjecture}
\begin{remark}  It may be that some sort of geometric regularity should be required for $T$.  Certainly, it ought to suffice for $T$ to be convex.	
\end{remark}

The remainder of this paper is organized as follows.   In \S\ref{sec:linetraps}, numerical analysis of the hitting probabilities for $T$ the line segment $[-1,1]\times\{0\}$ are presented in order to illustrate the phenomena discussed here and to motivate Conjecture \ref{conj:conj}.   In \S\ref{sec:CR} we briefly review the definition of the conformal radius, the Green's function of a planar domain with pole at infinity and some related complex analysis. The proofs of Theorems \ref{theorem:new}, \ref{theorem:Hunt} and Corollary \ref{cor:Hunt} are given in \S\ref{sec:proofs}. In two appendices,
\begin{enumerate}
\item[\S\ref{sec:hitting}.] The validity of the algorithm used in \S\ref{sec:linetraps} to simulate hitting times of $[-1,1]\times \{0\}$ is proved.
\item[\S\ref{sec:K0}.] Controlled asymptotics for the Bessel function $K_0$ required in \S\ref{sec:proofs} are stated and proved.
\end{enumerate}
 
\section{An example: line segment traps}\label{sec:linetraps}
Let us examine the approximations presented above if the trap consists of a single line segment.  By scaling and rotational symmetries, it suffices to consider the case $$T\ = \ [-1,1] \times \{0\}  \ = \setb{(x,0)}{|x| \le 1}.$$ 
No exact formula for $P_T(z,t)$ seems to be available, however in Proposition \ref{prop:conformalradius} below we show that the conformal radius of $T$ is $r_T=\nicefrac{1}{2}$.

Although there is not an exact formula for $p_T$, there is a very fast random algorithm for simulating the process of a Brownian particle released from an arbitrary point $(x,y)$ and stopped if it hits $T$ before an arbitrary cutoff time $t_{\mathrm{max}}$:
\begin{quotation}
\textsf{\begin{enumerate}
\item Let $(X,Y)=(x,y)$ and $S=0$.
\item While $(X,Y) \not \in T$ and $S \le t_{\mathrm{max}}$ repeat:
\begin{enumerate}
	\item  If $Y\neq 0$, let $$X \mapsto X +\frac{|Y|}{|g_1|} g_2 , \quad Y \mapsto 0, \quad S \mapsto S + \frac{Y^2}{g_1^2}, $$ with $g_1$ and $g_2$ independent standard normal random variables, independent from each other and from any variables used previously.
	\item If $Y = 0$, let $$X \mapsto \frac{X}{|X|}, \quad Y \mapsto \frac{|X| -1 }{|g_1|} g_2, \quad S \mapsto S + \frac{(|X|-1)^2}{g_1^2}$$ with $g_1$ and $g_2$  standard normal random variables, independent from each other and from any variables used previously.
\end{enumerate}
\item Return the final value of $(X,Y,S)$.
\end{enumerate}
}
\end{quotation}

The random value $S$ returned by the algorithm satisfies 
\begin{equation}\label{eq:algorithm} \Pr(S < t) \ = \ p_T(x,y,t)	
\end{equation}
for any $t\le  t_{\mathrm{max}}$ \tem \ i.e., $S$ has the same distribution as the minimum of $t_{\mathrm{max}}$ and $t_T(x,y)$, the first time that a Brownian motion started at $(x,y)$ hits $T$. 
Eq.\ \eqref{eq:algorithm} is a consequence of Theorem \ref{theorem:algorithm}, proved below in Appendix \ref{sec:hitting}.   The main points in the argument are:
\begin{enumerate}
\item a Brownian motion $(X(t),Y(t))$ started at a point $(x,y)$ with $y \neq 0$ must hit the line $Y=0$ before (or as) it hits $T$;
\item a Brownian motion $(X(t),Y(t))$ started at a point $(x,y)$ with $|x|>1$ must hit the line $X=\nicefrac{x}{|x|}$ before hitting $T$; and
\item the time for a planar Brownian motion to hit a given line has the same distribution as $\nicefrac{d^2}{g^2}$ where $d$ is the distance from the initial point to the given line and $g$ is a standard normal random variable.
\end{enumerate}

By counting the number of times $C(t)$ that the output $S < t$ in $N$ repetitions of this algorithm we can approximate
\begin{equation}\label{eq:pTapprox}p_T(x,y,t) \ = \ \Pr(S<t) \ \approx \ \ = \ \frac{C(t)}{N},	
\end{equation}
for $t \le t_{\mathrm{max}}$.
Furthermore, since each repetition constitutes an independent Bernoulli trial with success probability $p_T(x,y,t)$ we can estimate the error in this approximation by standard statistical methods. 

\begin{figure}
\includegraphics[width=\columnwidth]{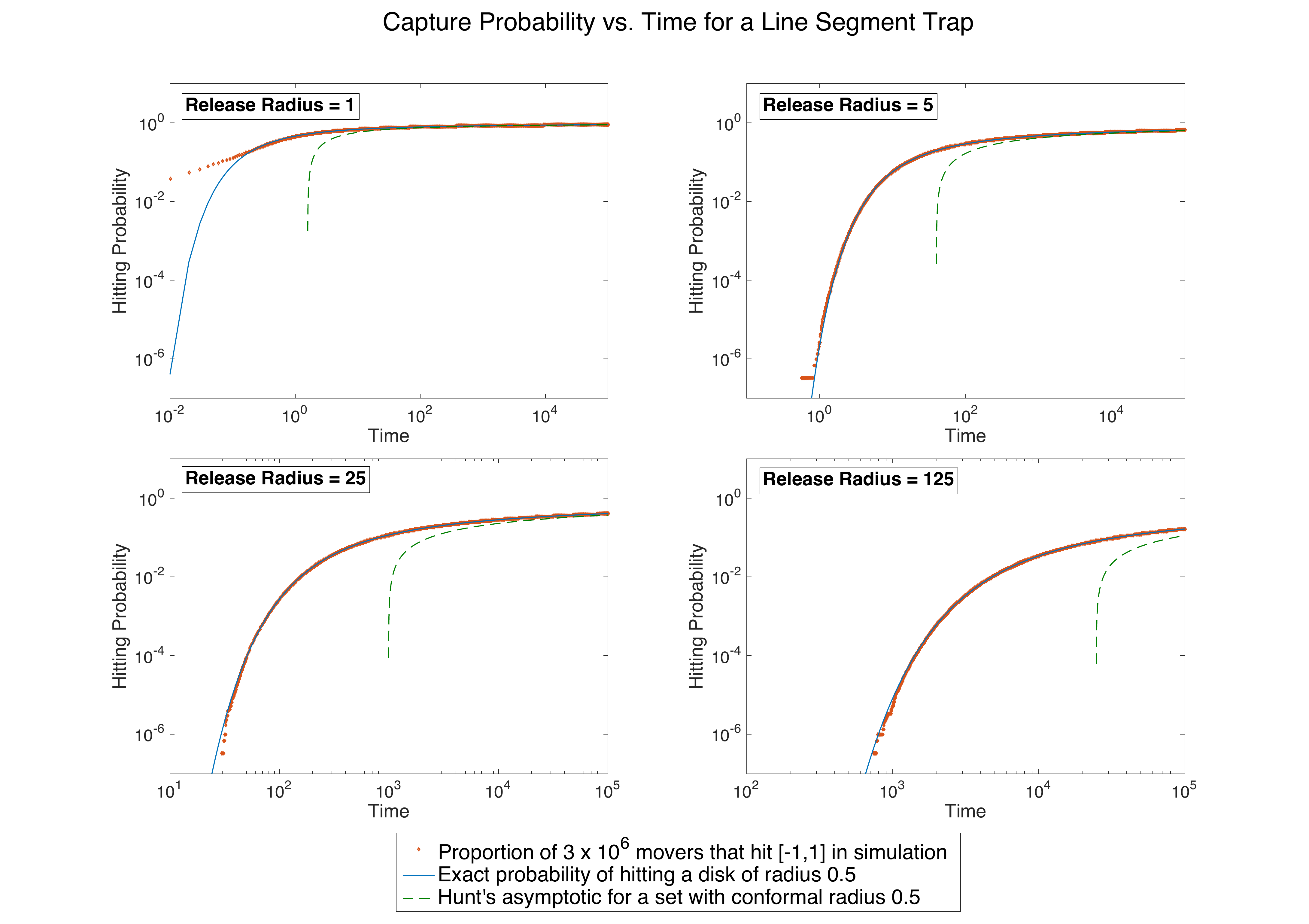}
\caption{\label{fig:1}Time evolution of proportion captured in simulations of hitting a line segment trap}
\end{figure}

The algorithm was implemented in Matlab and used to compute, for $r=1,$ $5$, $25$ and $125$, the time to hit $[-1,1]$ for $3\times 10^6$ movers initially distributed on a circle of radius $r$. Figure \ref{fig:1} shows, for each release radius $r$, the proportion $\mathrm{Prop}(t,r)$ of movers captured by the trap $T=[-1,1]$ plotted versus time, up to time $t=10^5$.  For comparison, each plot also shows the exact probability  $p^D_{0.5}(r,t)$ of hitting a disk of radius $r_T=\nicefrac{1}{2}$, computed from the formula \cite{Nicholson1921,Carslaw1940a,Wendel1980}:\footnote{Eq.\ \eqref{eq:integralformula} can be obtained from eq.\ \eqref{eq:diskf} by writing
$$P_{r_T}^D(r,t) \ = \ \frac{1}{2\pi } \int_{\Gamma} \e^{\zeta t} \frac{K_0(\sqrt{2 \zeta} r)}{\zeta K_0(\sqrt{2\zeta} r_T)} \di \zeta, $$
where $\Gamma$ is a contour of the form $\zeta = \lambda + \im y$, $-\infty < y < \infty$, and shifting the contour of integration.
}
\begin{equation}
p_{r_T}^D(r,t)=1+\frac{1}{\pi}\int_{0}^{\infty}\frac{1}{y} \frac{J_{0}\left(\nicefrac{yr}{r_T} \right) Y_0(y)-J_0(y)Y_0(\nicefrac{yr}{r_T})}{J_0(y)^2 + Y_0(y)^2}e^{-\frac{t y^{2}}{2 r_0}}dy,  \label{eq:integralformula}
\end{equation}
where $J_0$ and $Y_0$ are the order zero Bessel functions of the first and second kind.  Also shown is Hunt's asymptotic $1-\nicefrac{\ln \nicefrac{r}{r_T}}{\ln \nicefrac{\tau}{\tau_0}}$.  
One can clearly see that the approximation $\mathrm{Prop}(r,t)\approx p_{r_T}^D(r,t)$ is accurate over a much wider range of time scales than the asymptotic eq\ \eqref{eq:Huntapprox}.  In fact, for each release radius except for $r=1$, $p_{r_T}^D(r,t)$ is essentially indistinguishable from $\mathrm{Prop}(r,t)$ whenever the proportion is non-zero. For $r=1$ there is some deviation for very short times ($t <0.5$).  This is not surprising, since some individuals at this release radius start out extremely close to the trap and are captured very quickly.

\begin{figure}
\includegraphics[width=\columnwidth]{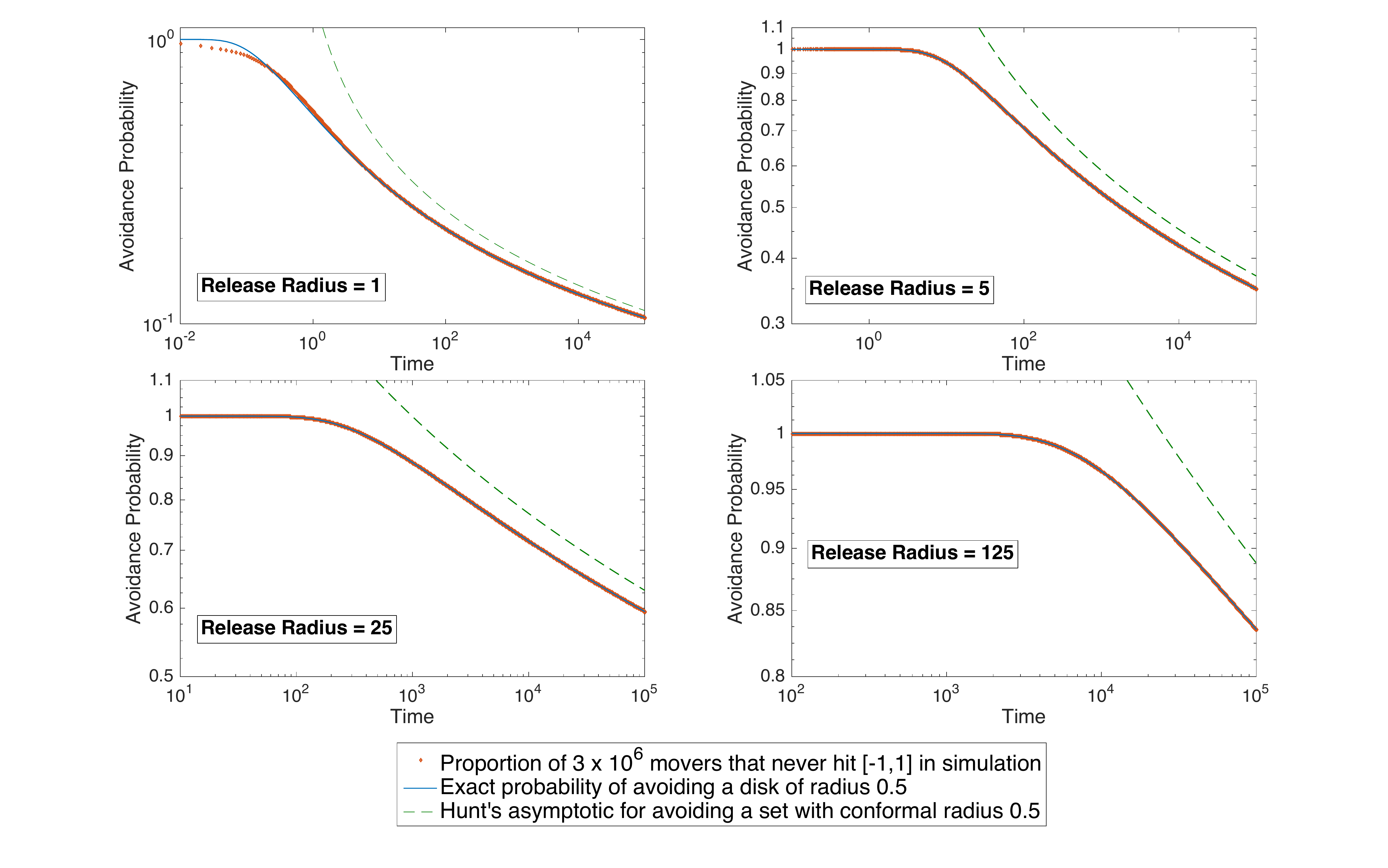}
\caption{\label{fig:2}Time evolution of proportion avoiding the trap in simulations of hitting a line segment trap}
\end{figure}
In Figure \ref{fig:2} the proportion avoiding the trap $1-\mathrm{Prop}(r,t)$ is shown for each release radius, with the corresponding exact probability  $1-p^D_{0.5}(r,t)$ of avoiding a disk of radius $\nicefrac{1}{2}$ and Hunt's asymptotic $\nicefrac{\ln \nicefrac{r}{r_T}}{\ln \nicefrac{\tau}{\tau_0}}$.  These are the same data as in Figure \ref{fig:1} plotted  so as to highlight the improvement over Hunt's asymptotic afforded by using $p_{0.5}^D(r,t)$ even at larger times when the hitting probability approaches one.
%Figure \ref{fig:2} shows the relative error 
%$$ \abs{1- \frac{\mathrm{Prop}(r,t)}{ p_{r_T}^D(r,t)}}$$
%for each of the four release radii.  Also shown is the  statistical error relative to $p_{r_T}^D(r,t)$:
%$$\frac{\Delta(r,t)}{p_{r_T}^D(r,t)},$$
%where $\Delta(r,t)$ is an estimated bound for
%$$\abs{p_T(r,t)-\mathrm{Prop}(r,t)} \le \Delta(r,t),$$  
%taken to be half the width of a 95\% confidence interval for the probability of success in a Bernoulli trial computed using the Wilson score method for an observed proportion $\mathrm{Prop}(r,t)$ of successes in $3\times 10^6$ trials.  For each of the larger
%release radii $r=5,25,$ and $125$, one sees that the error is dominated by statistical error (which would saturate at $6.4 \times 10^-7$ as $t\rightarrow \infty$, the confidence interval half-width for $3\times 10^7$ successes in $3\times 10^7$ trials). These plots provide some evidence that the relative error in the approximation
%$$\frac{p_T(r,t)-p_{r_T}^D(r,t)}{p_{r_T}(r,t)}$$
%converges to zero.  At the very least the actual relative error $E_T(t)$ in eq.\ \eqref{eq:conj} is bounded by the (quite small) statistical error in these simulations.

\section{The conformal radius of a planar region and the Green's function of its complement.}\label{sec:CR} In the remainder of the paper, it will be useful to identify the plane with the complex plane $$\mathbb{C} \ = \ \setb{z=x+\im y }{ x \text{ and } y \text{ are real numbers}},$$
where $\im $ is the imaginary unit ($\im ^2=-1$). If $T\subset \mathbb{C}$ is a connected proper subset that is neither empty nor a single point, then its unbounded complement in the Riemann sphere $\mathbb{C}\cup\{\infty\}$ is a proper simply connected set containing the point at infinity.  By the Riemann mapping theorem, there is a conformal map $\Phi_T:\mathbb{C} \setminus T \rightarrow \mathbb{C}\setminus \mathbb{D}$ fixing the point at infinity, where $\mathbb{D}$ is the unit disk.  Furthermore, this map is unique if we require it to have positive derivative at infinity.  We define the conformal radius $r_T$ to be the inverse of the corresponding derivative at infinity:\footnote{Note that $r_T$ is the unique radius of a disk $D$ such that there is a conformal map from $\mathbb{C}\setminus T$ to $\mathbb{C} \setminus D$ asymptotic to $z$ as $z\rightarrow \infty$.  Indeed $r_T \Phi_T(z)$ is such a map.}
$$r_T \ := \  \lim_{z\rightarrow \infty} \frac{z}{\Phi_T(z)}.$$

For example, we can compute the conformal radius of a line segment $T=[a,b]$:
\begin{proposition}\label{prop:conformalradius} The conformal radius of $[a,b]$, where $a < b$ are real numbers is $\nicefrac{b-a}{4}$.	
\end{proposition}
\begin{proof} By scaling and translational symmetry it suffices to consider $T=[-1,1]$. For this set we have the explicit conformal map
$$\Phi_T(z) \ = \	z + \sqrt{z^2 -1},$$
where we choose the branch of $\sqrt{z^2 -1}$  that it is positive for $z > 1$ and negative for $z < -1$.  One may easily check that $\Phi_T$ maps $\mathbb{C}\setminus [-1,1]$ onto $\mathbb{C} \setminus \mathbb{D}$ and that
$$\Phi_T^{-1}(w) \ =\ \frac{1}{2} \left ( w + \frac{1}{w} \right )$$
is the corresponding inverse map. Since  $$\lim_{z\rightarrow \infty} \frac{\Phi_T(z)}{z} \ = \ 2,$$
the result follows.
\end{proof}

Recall that the Green's function on $\mathbb{C}\setminus T$ with pole at $\infty$ is the unique harmonic function $H_T$ on $\mathbb{C}\setminus T$ such that
$$\lim_{z\rightarrow T} H_T(z) = 0 \quad \text{and} \quad H_T(z) \sim \frac{1}{\pi}\ln |z| \text{  as } z \rightarrow \infty.$$
Thus the Green's function for the unit disk $\mathbb{D}$ is
$$H_{\mathbb{D}}(z) \ = \ \frac{1}{\pi} \ln |z|.$$
Since the composition of a harmonic function with a holomorphic map is again harmonic, it follows that the Green's function for $\mathbb{C} \setminus T$ is
$$H_T(z) \ = \ H_{\mathbb{D}}\circ \Phi_T(z) \ = \  \frac{1}{\pi} \ln \abs{\Phi_T(z)}.$$
Note that $H_T$ satisfies 
\begin{equation}\label{eq:HTrT}\lim_{z\rightarrow \infty} H_T(z) - \frac{1}{\pi} \ln |z| \ = \ -\frac{1}{\pi}\ln r_T.	
\end{equation}

\section{Proofs}\label{sec:proofs}

 The function $F_T(z,\tau)$ solves the following boundary value problem
$$ \begin{cases} \frac{2}{\tau} F_T(z,\tau) \ = \ \Delta F_T(z,\tau) \ , & z \in \mathbb{C}\setminus T\\
	F_T(z,\tau)=1 \ , & z \in T \\
	\lim_{z\rightarrow \infty} F_T(z,\tau) = 0
 \end{cases}
 $$
where $\Delta$ denotes the Laplacian.  When $T$ is the disk $\{|z| \le r_T \}$ centered at the origin, the corresponding solution is a function only of $r=|z|$ and satisfies
$$ \begin{cases} \frac{2}{\tau} f_{r_T}^D(r,\tau) \ = \ \left (\partial_r^2 + \frac{1}{r}\partial_r \right ) f_{r_T}^D(r,\tau) \ , & r > r_T\\
	f_{r_T}^D(r,\tau)=1 \ , & r \le r_T \\
	\lim_{r\rightarrow \infty} f_{r_T}^D(r,\tau) = 0
 \end{cases} .$$
 After multiplying through by $r^2$, the resulting equation becomes the modified Bessel equation \cite[\S10.25]{DLMF}, with the unique solution for the given boundary data being
$$ f_{r_T}^D(r,\nicefrac{1}{\lambda}) \ = \ \frac{K_0(\sqrt{2\lambda} r)}{K_0(\sqrt{2\lambda} r_T)}, \quad r > r_T.$$
Here and below we let $\lambda = \nicefrac{1}{\tau}$.

We will derive estimates on $F_T$ by relating it to resolvents of the Laplacian in $\mathbb{C}\setminus T$. Let  $p_T(z,w,t)$ denote the probability density  at time $t$ for Brownian motion started at $z$ and terminated upon hitting $T$. Let 
$$p(z,w,t) \ = \ \frac{1}{2\pi t} \e^{-\nicefrac{|z-w|^2}{2 t}}$$
 denote the corresponding Gaussian distribution for Brownian motion in the full plane.  Note that $P_T(z,\cdot,t)$ is a sub-probability density, with total mass equal to the survival probability:
$$ \int_{\mathbb{C}\setminus T} p_T(z,w,t) \di^2 w \ = \ 1 - F_T(z,t),$$
where $\di^2 w$ denotes area measure. The corresponding resolvents for $\mathbb{C}\setminus T$ and for $\mathbb{C}$ are
$$G_T(z,w,\lambda) \ = \ \int_0^\infty \e^{-\lambda t}p_T(z,w,t) \di t \qquad \text{and} \qquad G(z,w,\lambda) \ = \ \int_0^\infty \e^{-\lambda t}p(z,w,t) \di t, $$ respectively.
Note that
$$G(z,w,\lambda) \ = \ \frac{1}{\pi} K_0(\sqrt{2\lambda}|z-w|),$$
by \cite[Eq.\ 10.32.10]{DLMF}.

Let $\Evc{\cdot}{\Brown(0)=z}$ denote expectation with respect to Brownian motion starting at $z$.  For a given Brownian path  $\Brown$, let
$$\tau_T(\Brown) = \inf \setb{t}{\Brown(t) \in T}$$
denote the first time $\Brown$ hits $T$
and let $$\zeta_T(\Brown) \ = \ \Brown(\tau_T(\Brown))$$ 
denote the corresponding point in $\partial T$ hit by $\Brown$.  Note that 
$$P_T(z,t) \ = \ \Evc{\1[\tau_T(\Brown) <t]}{\Brown(0)=z}.$$
It follows that
\begin{equation}\label{eq:Fformula}F_T(z,\nicefrac{1}{\lambda}) \ = \ \Evc{\e^{-\lambda \tau_T(\Brown)}}{\Brown(0)=z}.	
\end{equation}

Given a set $A$ of positive measure in $\mathbb{C}\setminus T$ at a positive distance from $z$, we have 
\begin{multline*}\int_A  G_T(z,w,\lambda) \di^2 w \ = \ \int_0^\infty \e^{-\lambda t} \int_A  p_T(z,w,t) \di^2 w \di t  \\ = \ \Evc{\int_0^{\tau_T(\Brown)} \e^{-\lambda t}\1_A(\Brown(t))\di t}{\Brown(0)=z}	
\end{multline*}
and thus
\begin{multline*}\int_A G_T(z,w,\lambda) \di^2 w \\ = \ \Evc{\int_0^{\infty } \e^{-\lambda t}\1_A(\Brown(t))\di t}{\Brown(0)=z} - \Evc{\int_{\tau_T(\Brown)}^\infty \e^{-\lambda t}\1_A(\Brown(t))\di t}{\Brown(0)=z} \\
= \ \int_A \left [ G(z,w,\lambda) \ - \ \Evc{\e^{-\lambda \tau_T(\Brown)} G(\zeta_T(\Brown),w,\lambda)}{\Brown(0)=z} \right ] \di^2 w	,
\end{multline*}
 using the strong Markov property for Brownian motion. Taking $A$ to be a disk, dividing by the measure $|A|$ of $A$ and taking the limit as $|A|\rightarrow 0$, we obtain
 \begin{equation}\label{eq:GTrep}G_T(z,w,\lambda) \ = \ \frac{1}{\pi} \left [ K_0(\sqrt{2\lambda}|z-w|) - \Evc{\e^{-\lambda\tau_T(\Brown)} K_0(\sqrt{2\lambda}|\zeta_T(\Brown)-w|)}{\Brown(0)=z} \right ].\end{equation}
Since $G_T(z,w,\lambda)$ vanishes as $w \ra \partial T$, we obtain
\begin{equation}\label{eq:K0-K0}    K_0(\sqrt{2\lambda}|z-w|)  \ = \ \Evc{\e^{-\lambda\tau_T(\Brown)}   K_0(\sqrt{2\lambda}|\zeta_T(\Brown)-w|)}{\Brown(0)=z}, \quad w\in \partial T .	
\end{equation}
This identity is the starting point for the proofs of both Theorem \ref{theorem:new} and Theorem \ref{theorem:Hunt}.

\subsection{Relating the Green's function $H_T$ to Brownian motion.}
The limit of $G_T$ as $\lambda \rightarrow 0$,
$$\lim_{\lambda \rightarrow 0} G_T(z,w,\lambda) \ := \ G_T(z,w,0) \ = \ \int_0^\infty p_T(z,w,t) \di t$$
is called the \emph{Green's function for $\mathbb{C}\setminus T$ with pole at $w$.} The following formula of Hunt \cite{Hunt1956} relates $G_T(\cdot,\cdot,\lambda)$ to the Green's function with pole at infinity:\footnote{To prove eq.\ \eqref{eq:Huntsid}, one verifies that the function on the right hand side is a harmonic function of $z$, vanishes for $z\in \partial T$ and is asymptotic to $\ln |z|$ as $z \rightarrow \infty$. It follows that the right hand side is equal to $H_T$ \tem \ and thus is in fact independent of $w$!}
 \begin{equation}\label{eq:Huntsid}H_T(z) \ = \ G_T(z,w,0) + \frac{1}{\pi} \ln |z-w| - \frac{1}{\pi}\Evc{\ln |\zeta_T(\Brown)-w|}{\Brown(0)=z}.	
\end{equation}
Taking $w \rightarrow \partial T$ in eq.\ \eqref{eq:Huntsid} we conclude that
\begin{equation}H_T(z) \ = \ \frac{1}{\pi} \ln |z-w| - \frac{1}{\pi} \int_{\partial T} \ln \abs{\zeta - w} \di \mu_z(\zeta), \qquad w\in \partial T,\label{eq:Huntsid2}	
\end{equation}
where $\mu_z$ denotes the distribution of $\zeta_T(\Brown)$ for $\Brown(0)=z$.  As $z\rightarrow \infty$, the measures $\mu_z$ converge vaguely to a probability measure $\mu=\mu_\infty$ on $\partial T$, which is the distribution of the ``first point of $T$ hit by a Brownian motion starting at $\infty$,'' a notion that can be made precise by considering Brownian motion on the Riemann sphere. Taking $z \rightarrow \infty$ and using eq.\ \eqref{eq:HTrT} we conclude that \begin{equation}\label{eq:measure} \int_{\partial T} \ln |\zeta -w| \di \mu(\zeta) \ = \ \ln r_T ,
\end{equation}
for any $w \in \partial T$.  Furthermore, we obtain another formula for $H_T(z)$ namely
\begin{equation}\label{eq:HTrep2} H_T(z) \ = \ \frac{1}{\pi} \int_{\partial T} \ln |z-\zeta| \di \mu(\zeta) - \frac{1}{\pi}\ln r_T.	
\end{equation}
To see that this formula holds, integrate eq.\ \eqref{eq:Huntsid2} with respect to $\di \mu(w)$ and apply eq.\ \eqref{eq:measure}. (Alternatively, note that the function on the right hand side is harmonic in $z$, vanishes on $\partial T$ and has the correct asymptotic as $z \rightarrow \infty$, so it is indeed $H_T$.)

\subsection{Proof of Theorem \ref{theorem:new}}
By eq.\ \eqref{eq:K0-K0} we have 
\begin{multline}\label{eq:startingpoint}\frac{1}{2\pi r} \int_{|z|=r} \int_{\partial T}  K_0(\sqrt{2\lambda} |z-w|) \di \mu(w) |\di z| \\ = \  \frac{1}{2\pi r} \int_{|z|=r} \int_{\partial T}  \Evc{\e^{-\lambda \tau_T(\Brown)} K_0(\sqrt{2\lambda}|\zeta_T(\Brown)-w|)}{\Brown(0)=z}  \di \mu(w) |\di z|.\end{multline}

Let us first look at the left hand side of eq.\ \eqref{eq:startingpoint}. For fixed $w\in \mathbb{C} \setminus \{0\}$, $$g(r,w;\lambda) \ := \ \frac{1}{2\pi r} \int_{|z|=r} K_0(\sqrt{2\lambda} |z-w|)|\di z|.$$
solves
$$2\lambda g(r,w;\lambda) \ = \ \partial_r^2 g(r,w;\lambda) + \frac{1}{r}\partial_r g(r,w;\lambda),\quad r\neq |w|.$$
Furthermore $g$ is continuous as a function of $r$ and converges to $K_0(\sqrt{2\lambda}|w|)$ as $r\ra 0$. It follows that
$$g(r,w;\lambda) \ = \ \begin{cases}
 K_0(\sqrt{2\lambda} |w|) I_0(	\sqrt{2\lambda} r) & r < |w| \ , \\
 K_0(\sqrt{2\lambda} r ) I_0(\sqrt{2\lambda} |w|) & r > |w| \ ,
 \end{cases}$$
 where $I_0$ denotes the order zero modified Bessel function of the first kind \tem \ see \cite[\S 10.25]{DLMF}.
Thus
\begin{equation}\label{eq:LHS}\int_{\partial T}\frac{1}{2\pi r} \int_{|z|=r} K_0(\sqrt{2\lambda} |z-w|)|\di z| \di \mu(w) \ = \ \left [ \int_{\partial T} I_0(\sqrt{2\lambda} |w|) \di \mu(w) \right ] K_0(\sqrt{2\lambda} r)	
\end{equation}
provided $r > r_0 := \max_{w\in \partial T} |w|$.

We now turn to the right hand side of eq.\ \eqref{eq:startingpoint}.  Using the upper bound \eqref{eq:K0upperboundM=0} for $K_0$ proved below and eq.\ \eqref{eq:Fformula}, we have
\begin{multline*}
	\int_{\partial T} \Evc{\e^{-\lambda \tau_T(\Brown)} K_0(\sqrt{2\lambda}|\zeta_T(\Brown)-w|)}{\Brown(0)=z} \di \mu(w) \\ 
\le  \ \Bigg ( - \ln \sqrt{\frac{\lambda}{2}} - \gamma - \ln r_T  \  + \ 0.79 \ \sup_{\zeta\in \partial T} \int_{\partial T} \frac{\lambda |w-\zeta|^2 }{2} \abs{\ln \frac{\lambda |w-\zeta|^2 }{2}} \di \mu(w)\Bigg )  F_T(z,\nicefrac{1}{\lambda}),
\end{multline*} 
provided $\sqrt{2\lambda} \diam(T)  < 2 \e^{-\gamma}$.  Because $x\mapsto x |\ln x|$ is monotone increasing on $[0,\nicefrac{1}{e}]$, we find that if $\lambda  \le \nicefrac{2}{\e d^2}$ with $d\ge \diam(T)$, then 
\begin{multline*} \int_{\partial T} \Evc{\e^{-\lambda \tau_T(\Brown)} K_0(\sqrt{2\lambda}|\zeta_T(\Brown)-w|)}{\Brown(0)=z} \di \mu(w)
\\ \le \ \left ( - \ln \sqrt{\frac{\lambda}{2}} - \gamma - \ln r_T + 0.79 \ \frac{\lambda d^2}{2} \abs{\ln \frac{\lambda d^2}{2} }\right ) F_T(z,\nicefrac{1}{\lambda}) \\  \le \ \left ( K_0(\sqrt{2\lambda} r_T) + 0.79 \ \frac{\lambda d^2}{2} \abs{\ln \frac{\lambda d^2}{2} }\right ) F_T(z,\nicefrac{1}{\lambda}),
\end{multline*}
where,  in the last step, we have used the lower bound \eqref{eq:K0lowerboundM=0} for $K_0$.
Similarly, using (\ref{eq:K0lowerboundM=0}, \ref{eq:K0upperboundM=0}) in the reverse order we obtain the following lower bound
 \begin{multline*}
	\int_{\partial T} \Evc{\e^{-\lambda \tau_T(\Brown)} K_0(\sqrt{2\lambda}|\zeta_T(\Brown)-w|)}{\Brown(0)=z} \di \mu(w) \\ \ge  \ \left ( -\ln \sqrt{\frac{\lambda}{2}} - \gamma - \ln r_T \right ) F_T(z,\nicefrac{1}{\lambda})  \\ \ge \ \left ( K_0(\sqrt{2\lambda} r_T) - 0.79 \ \frac{\lambda d^2}{2} \abs{\ln \frac{\lambda d^2}{2} }\right ) F_T(z,\nicefrac{1}{\lambda}),
\end{multline*}
provided $\e^{\gamma} r_T \le d$ and $\lambda < \nicefrac{2}{ed^2}$.

We take $d = \max(\e^{\gamma} r_T, \diam T)$ and $\lambda < \nicefrac{2}{ed^2}$ so that both the upper and lower bounds hold.  Then 
 \begin{multline}\label{eq:RHS}
	\frac{1}{2\pi r} \int_{|z|=r} \int_{\partial T} \Evc{\e^{-\lambda \tau_T(\Brown)} K_0(\sqrt{2\lambda}|\zeta_T(\Brown)-w|)}{\Brown(0)=z} \di \mu(w)|\di z| \\
	= \ (1 + A_T(r)) K_0(\sqrt{2\lambda} r_T)\frac{1}{2\pi r} \int_{|z|=r} F_T(z,\nicefrac{1}{\lambda})|\di z|
	\end{multline}
	with
\begin{equation}\label{eq:ATbound}\abs{A_T(r)} \ \le \ 0.79 \ \frac{\lambda d^2 \abs{\ln \frac{\lambda d^2}{2} }}{2 K_0(\sqrt{2\lambda} r_T)} \ \le \ 0.79 \ \lambda d^2 \frac{\abs{\ln \frac{\lambda d^2}{2}}}{\abs{\ln \frac{\lambda r_T^2}{2} + 2 \gamma} } \ \le \ 0.79 \ \lambda d^2 , 	
\end{equation}
where we have used eq.\ \eqref{eq:K0lowerboundM=0} and the fact that   $1 \ge \nicefrac{\lambda d^2}{2} \ \ge  \ \nicefrac{\lambda r_T^2 \e^{2\gamma}}{2} $.

Putting together eqs.\ (\ref{eq:startingpoint}, \ref{eq:LHS}, \ref{eq:RHS}), we find the following formula for $f_T = \frac{1}{2\pi r}\int_{|z|=r}F_T$:
$$f_T(r,\nicefrac{1}{\lambda}) \ = \
\left [ \frac{ \int_{\partial T} I_0(\sqrt{2\lambda} w) \di \mu(w)}{1 + A_T(r)} \right ] 	
\frac{K_0(\sqrt{2\lambda} r)}{K_0(\sqrt{2\lambda} r_T)} .$$
Since $I_0(x)\ge 0$ for $x\ge 0$, $I_0(0)=1$ and $I_0'(0)=0$, we have 
$$ 0 \ \le \ \int_{\partial T} I_0(\sqrt{2\lambda} |w|) \di \mu(w) - 1 \ \le \ \frac{\lambda d^2 }{2}  \sup_{0\le x\le \sqrt{2\lambda} d} I_0''(x)  \\ \le \  \frac{\lambda d^2}{2}  \frac{I_0(\sqrt{2\lambda} d) + I_2(\sqrt{2\lambda} d)}{2},
$$
where $I_2$ denotes the order two modified Bessel function of the first kind, and we have used the identity \cite[\S 10.29]{DLMF}
$$2 I_0''(x) =  I_0(x) + I_2(x) $$
as well as the fact that $I_0$ and $I_2$ are increasing.
Since $\sqrt{2\lambda} d < \nicefrac{2}{\sqrt{\e}}$, we have
$$ \int_{\partial T} I_0(\sqrt{2\lambda} |w|) \di \mu(w) - 1 \ \le \   \frac{I_0(\nicefrac{2}{\sqrt{\e}}) + I_2(\nicefrac{2}{\sqrt{\e}})}{4} \lambda^2 d\ = \ \left ( 0.4027\cdots  \right ) \lambda^2 d \ \le \ 0.41 \ \lambda^2 d . $$
With this inequality and \eqref{eq:ATbound} we conclude that
$$
\abs{	f_T(r,\tau) - f_{r_T}^D(r,\tau)} \ \le \  \frac{0.41 + 0.79}{1 -  0.79 \ \frac{d^2}{\tau}} \frac{d^2}{\tau} f_{r_T}^D(r,\tau) \ \le \  2.9 \ \frac{  d^2}{\tau} f_{r_T}^D(r,\tau),
$$
since  $\nicefrac{d^2}{\tau} \ \le \ \nicefrac{2}{\e}$. \qed

\subsection{Proof of Theorem \ref{theorem:Hunt}}
By eq.\ \eqref{eq:K0-K0}  we have
\begin{equation}\label{eq:startingpoint2}    0 \ = \ \int_{\partial T} \left [ K_0(\sqrt{2\lambda}|z-w|) -  \Evc{\e^{-\lambda\tau_T(\Brown)}   K_0(\sqrt{2\lambda}|\zeta_T(\Brown)-w|)}{\Brown(0)=z} \right ] \di \mu(w) .	
\end{equation}
Using eqs.\ (\ref{eq:K0lowerboundM=0}, \ref{eq:K0upperboundM=0}) below and eq.\ \eqref{eq:Fformula} we conclude
\begin{multline*}
	0 \ge - \ln \sqrt{\frac{\lambda}{2}} -\gamma  - \int_{\partial T} \ln |z-w| \di \mu(w) 
	 + \left (\ln \sqrt{\frac{\lambda}{2}} + \gamma \right ) F_T(z,\nicefrac{1}{\lambda}) \\ + \Evc{\e^{-\lambda \tau_T(\Brown)}  \int_{\partial T} \ln |\zeta_T(\Brown)-w| \di \mu(w)}{\Brown(0)=z} \\ - 0.79 \  \sup_{\zeta\in \partial T} \int_{\partial T} \frac{\lambda |w-\zeta|^2 }{2} \abs{\ln \frac{\lambda |w-\zeta|^2 }{2}} \di \mu(w).
\end{multline*}
Thus, by eqs. (\ref{eq:measure}, \ref{eq:HTrep2}),
$$
0 \ge - \pi H_T(z) - \left ( \ln \sqrt{\frac{\lambda}{2}} + \gamma + \ln r_T \right ) \left ( 1 - F_T(z,\nicefrac{1}{\lambda}) \right )   - 0.79 \   \frac{\lambda d^2 }{2} \abs{\ln \frac{\lambda d^2 }{2}} ,
$$
where as above the bound on the final term uses the fact $x \mapsto x \abs{\ln x}$ is increasing for $x < \e^{-1}$.  In this way we obtain
$$F_T(z,\tau) \ \ge \ 1 - \frac{2\pi H_T(z)}{\ln \nicefrac{\tau}{\tau_0}} - \ 0.79 \ \frac{d^2}{\tau} \frac{\ln \nicefrac{2 \tau}{d^2}}{\ln \nicefrac{\tau}{\tau_0}}, $$
with $\tau_0 \ = \ \frac{1}{2}\e^{2\gamma} r_T^2$.  If $\tau > \nicefrac{d^2}{2}$, then $1 \le  \nicefrac{2 \tau}{d^2} \ \le \ \nicefrac{\tau}{\tau_0}$.  Thus $\ln \nicefrac{2\tau}{d^2}  \le \ln \nicefrac{\tau}{\tau_0}$ and the claimed lower bound \eqref{eq:Huntlower} holds.

In the other direction, if $\lambda < \nicefrac{2}{R_z^2 \e }$, with
$$R_z = \max (\sup_{w\in \partial T} |z-w|, r_T\e^\gamma),$$
then
\begin{equation*}
\begin{split}
	0 \ \le & \ - \ln \sqrt{\frac{\lambda}{2}} -\gamma  - \int_{\partial T} \ln |z-w| \di \mu(w) 
	 + \left (\ln \sqrt{\frac{\lambda}{2}} + \gamma \right ) F_T(z,\nicefrac{1}{\lambda}) \\ & \  + \Evc{\e^{-\lambda \tau_T(\Brown)} \int_{\partial T} \ln |\zeta_T(\Brown)-w| \di \mu(w)}{\Brown(0)=z}  \ + 0.79 \ \int_{\partial T} \frac{\lambda |z-w|^2}{2} \abs{\ln  \frac{\lambda |z-w|^2}{2}} 
	  \\
	 \ \le & \ - \pi H_T(z) - \left ( \ln \sqrt{\frac{\lambda}{2}} + \gamma + \ln r_T \right ) \left ( 1 - F_T(z,\nicefrac{1}{\lambda} \right ) \ + \  0.79 \ \frac{\lambda R_z^2}{2} \abs{\ln \frac{\lambda R_z^2}{2}},
	 \end{split}
	 \end{equation*}
from which the upper bound \eqref{eq:Huntupper} follows in a similar way. \qedhere

\subsection{Proof of Corollary \ref{cor:Hunt}}
Since $P_T(z,t)$ is an increasing function of $t$ we have
\begin{multline*}
P_T(z,t) \ \le \ \frac{\e^{\nicefrac{t}{\tau}}}{\tau} \int_t^\infty \e^{-\nicefrac{s}{\tau}} P_T(z,s) \di s \ \le \ \frac{1}{\tau} \int_0^\infty 	\e^{-\nicefrac{s}{\tau}} P_T(z,s) \di s + \frac{\e^{\nicefrac{t}{\tau}}-1}{\tau} \int_t^\infty \e^{-\nicefrac{s}{\tau}} \di s \\
= \ F_T(z,\tau) + 1 - \e^{-\nicefrac{t}{\tau}} \ \le \ F_T(z,\tau) + \frac{t}{\tau}
\end{multline*}
for any $\tau >0$.  Taking $\tau = s t$, and assuming $s > \nicefrac{\e R_z^2}{2 t}$, we have
$$P_T(z,t) \ \le \ 1 - \frac{2\pi H_T(z)}{\ln \nicefrac{t}{\tau_0} + \ln \nicefrac{s}{\tau_0}} + \frac{1}{s} + 0.8 \ \frac{R_z^2}{s t} \ . $$
If $t$ and $s$ are both larger than $\tau_0$, this is further bounded by
$$P_T(z,t) \ \le \ 1 - \frac{2\pi H_T(z)}{\ln \nicefrac{t}{\tau_0}}  + 2 \pi H_T(z)\frac{\ln \nicefrac{s}{\tau_0}}{\ln^2 \nicefrac{t}{\tau_0}} + \frac{1}{s} + 0.8 \ \frac{R_z^2}{s t}   .$$
 Optimizing over the choice of $s$ leads to
$$s = \frac{1 + 0.8 \nicefrac{R_z^2}{t}}{2\pi H_T(z)} \ln^2 t/\tau_0,$$
and thus
$$P_T(z,t) \ \le \ 1 - \frac{2 \pi H_T(z)}{\ln \nicefrac{t}{\tau_0}} + \frac{\ln \ln^2\nicefrac{t}{\tau_0} }{\ln^2\nicefrac{t}{\tau_0}} \left [ 2 \pi H_T(z)  + o(1) \right ].$$

Estimating $P_T(z,t)$ from below we have
\begin{multline*} P_T(z,t) \ \ge \ (1-\e^{-\nicefrac{t}{\tau}} ) P_T(z,t) \ \ge \ \frac{1}{\tau}\int_0^t\e^{-\nicefrac{s}{\tau}} P(z,s) \di s \\ \ge \ F_T(z,\tau) - \frac{1}{\tau} \int_t^\infty  \e^{-\nicefrac{s}{\tau}} \di s \ = \ F_T(z,\tau) -\e^{-\nicefrac{t}{\tau}}	
\end{multline*}
for any $\tau >0$.  Taking $\tau = \nicefrac{t}{\ln^2 \nicefrac{t}{\tau_0}} $ we have
\begin{multline*}P_T(z,t) \ \ge \ 1 - \frac{2\pi H_T(z)}{\ln \nicefrac{t}{\tau_0} }\frac{1}{1- \frac{\ln \ln^2 \nicefrac{t}{\tau_0}}{\ln \nicefrac{t}{\tau_0}}}  - \e^{- \ln^2 \nicefrac{t}{\tau_0}} - 0.9 \  \frac{d^2 \ln^2 \nicefrac{t}{\tau_0} }{t} \\
\ge \ 1 - \frac{2\pi H_T(z)}{\ln \nicefrac{t}{\tau_0} } - \frac{\ln \ln^2\nicefrac{t}{\tau_0}}{\ln^2\nicefrac{t}{\tau_0}} \left [ 2\pi H_T(z)   + o(1) \right ] ,
\end{multline*}
completing the proof. \qedhere

\appendix
\section{Simulating hitting of a line segment trap}\label{sec:hitting}
The validity of the algorithm described in \S\ref{sec:linetraps} is a consequence of the following
\begin{theorem}\label{theorem:algorithm}
	Let $(x_0,y_0)$ be a point in the plane, et $t_0=0$ and let 
	$$ g_1, g_2, \ldots \quad \text{and} \quad g_1', g_2',\ldots $$
	be two mutually independent sequences of independent standard normal random variables.
	 Construct a sequence of triples $\left[ (x_n,y_n,t_n)\right ]_{n=0}^\infty$ recursively as follows.  For $n=1,\ldots$,
	\begin{enumerate}
		\item if $|x_{n-1}|\le 1$ and $y_{n-1} =0$, let $$ (x_n,y_n,t_n) \ = \ (x_{n-1},y_{n-1},t_{n-1});$$
		\item if $y_{n-1} \neq 0$, let 
			$$ x_n = x_{n-1} +\frac{|y_{n-1}|}{|g_n|} g_n', \quad y_n = 0, \quad \text{and} \quad t_n = t_{n-1} + \frac{y_{n-1}^2}{g_n^2};$$
		\item if $y_{n-1} = 0$ and $|x_{n-1}| > 1$, let
			$$ x_n = \frac{x_{n-1}}{|x_{n-1}|}, \quad y_n = y_{n-1} + \frac{|x_{n-1}|-1}{|g_n|} g_n', \quad \text{and} \quad t_n = t_{n-1} + \frac{(|x_{n-1}-1|)^2}{g_n^2}.$$
	\end{enumerate}
	Let 
	$$ n_0 \ = \ \min \setb{n}{|x_n|\le 1 \text{ and } y_n =0}.$$
	Then $n_0 < \infty$ with probability one and $(t_{n_0},x_{n_0},y_{n_0})$ has the same distribution as $(t_T(x_0,y_0),$ $\zeta_T(x_0,y_0))$, where $t_T(x_0,y_0)$ denotes the time at which a standard Brownian motion $B$ started at $(x_0,y_0)$ hits $T = [-1,1]\times \{0\}$ and  $$\zeta_T(x_0,y_0) \ = \ B(t_T(x_0,y_0))$$ denotes the corresponding point of $T$ hit at time $t_T(x_0,y_0)$. 
\end{theorem}

\begin{proof} 

For a standard one dimensional brownian motion $b(t)$ started at $0$, the first time $t_x$ at which $b(t_x)=x$, for $x\neq 0$, satisfies the well known identity \cite[Theorem 2.21]{Brownian}:
$$\Pr( T_x \le t) \ = \ 2 *\Pr(b(t) \ge |x|) \ = \ 2 \int_{|x|}^\infty \frac{1}{\sqrt{2\pi t }} \e^{-\frac{y^2}{2t}} \di y \ = \ 2 \int_{\nicefrac{|x|}{\sqrt{t}}}^\infty  \frac{1}{\sqrt{2\pi}} \e^{-\frac{y^2}{2}} \di y.$$
Thus
$$\Pr( s \sqrt{T_x} \le |x|  ) \ = \ \Pr (g \ge s)$$
where $g$ is a standard normal random variable. It follows that $T_x$ has the same distribution as $ \nicefrac{x^2}{g^2}$. 

For planar Brownian motion $B$, with $B(0)= (x, y)$, it follows that the first time $B$ hits a given line $L$ is distributed as $ \nicefrac{d^2}{g^2}$ where $d=\dist((x,y),L)$ and $g$ is a standard normal.  Furthermore, at time $t=t_L(x,y)$ we have $B(t)=P_L(x,y) + \sqrt{t_X(x,y)} g' \wh{n}_L$ where $P_L(x,y)$ is the orthogonal projection of $(x,y)$ onto $L$, $\wh{n}_L$ is a unit vector parallel to $L$ and $g'$ is a standard normal independent of $g$.

Suppose $y_0>0$.  We see that $(x_1,y_1)$ has the same distribution as the first point of the $X$-axis hit by $B$, while $t_1$ has the same distribution as the first hitting time of the $X$-axis. If it happens that $|x_1|\le 1$ (so $(x_1,y_1)\in T$) then  we have $x_n=x_1$, $y_n=0$ and $t_n=t$ for all subsequent $n$.

On the other hand if $y_0 = 0$ then there are two cases $|x_0|\le 1$ or $|x_0|> 1$.  In the first case $(x_n,y_n,t_n)=(x_0,0,0)$ for all $n$ corresponding to the fact that the initial point is already in the trap $T$.  In the second case, in order for the Brownian motion $B$ to hit $T$ it must first hit the line $x= \nicefrac{x_0}{|x_0|}$, since this line separates the initial point $(x_0,0)$ from $T$.  The point $(x_1,y_1)$ has the distribution of the first point of this line hit by $B$, while $t_1$ has the distribution of the first hitting time.  Note that with probability one $y_1 \neq 0$.

By the strong Markov property of Brownian motion,  the distribution of $B(t+t_1)$ is the same as a new Brownian motion starting at the point $(x_1,y_1)$.  Repeating the above analysis we see that up until the point when $(x_n,y_n)$ is in $T$, the points $(x_n,y_n)$ alternate between the $X$-axis and  the lines $x=\pm 1$ and share the distribution of Brownian motion at the corresponding stopping times. The result follows since Brownian motion hits $T$ with probability one.
\end{proof}

\section{Controlled asymptotics for $K_0$}\label{sec:K0} The order zero modified Bessel function $K_0$ has the form \cite[\S10.31]{DLMF}
$$ K_0(x) \ = \ -  \left ( \ln \frac{x}{2} + \gamma  \right ) I_0(x)  + \Psi(x)$$
where $I_0(x)$ is the order zero modified Bessel function of the first kind
$$I_0(x) \ = \ \sum_{n=0}^\infty \frac{1}{n!^2} \left ( \frac{x^2}{4} \right )^n \ ,$$
and $\Psi$ has a convergent power series
$$\Psi(x)  \ = \ \sum_{n=1}^\infty \frac{h_n}{n!^2} \left( \frac{x^2}{4} \right )^n ,$$
with $h_n$  the $n$-th harmonic number
$$h_n \ = \ 1 + \frac{1}{2} + \cdots \frac{1}{n}.$$
The following Theorem gives bounds on the error that results from truncating these series.
\begin{theorem}\label{theorem:K0theorem} 
The following estimates hold for $K_0$.  For any $M=0,1,2,\ldots$:
\begin{enumerate}
\item If $x>0$, then
\begin{equation}\label{eq:K0lowerbound}
	K_0(x) \ \ge \ - \left (  \ln \frac{x }{2} + \gamma \right ) - \sum_{n=1}^{M} \frac{1}{n!^2} \left ( \frac{x^2}{4} \right )^n \left (  \ln \frac{x }{2} + \gamma -h_n \right )  ;
\end{equation}
\item If $x < 2\e^{-\gamma}$, then
 \begin{multline}\label{eq:K0upperbound}
	 K_0(x) \ \le \ - \left (  \ln \frac{x }{2} + \gamma \right ) - \sum_{n=1}^{M} \frac{1}{n!^2} \left ( \frac{x^2}{4} \right )^n \left (  \ln \frac{x }{2} + \gamma -h_n \right )  \\
	+  \ \frac{I_0(x)}{2\pi}\, \frac{\e^{2M+2}}{M+1} \left ( \frac{x}{2(M+1)} \right )^{2M+2} \abs{\ln \frac{x}{2(M+1)}}.
\end{multline}
\item If $2\e^{-\gamma} \le  x < 2\e^{h_M-\gamma}$, then
\begin{multline}\label{eq:K0upperboundbigx}
	 K_0(x) \ \le \ - \left (  \ln \frac{x }{2} + \gamma \right ) - \sum_{n=1}^{M} \frac{1}{n!^2} \left ( \frac{x^2}{4} \right )^n \left (  \ln \frac{x }{2} + \gamma -h_n \right )  \\
	+  \ \frac{I_0(x)}{2\pi}\, \frac{\e^{2M+2}(\gamma+\ln (M+1))}{M+1} \left ( \frac{x}{2(M+1)} \right )^{2M+2} .
\end{multline}
\end{enumerate}
\end{theorem}
\begin{remark} In the proofs of Theorems \ref{theorem:new} and \ref{theorem:Hunt}, we have used only the case $M=0$ of these bounds, for which the  estimates imply
\begin{equation}\label{eq:K0lowerboundM=0}
- \left (\ln \frac{x}{2} + \gamma \right ) \ \le \ K_0(x)	
\end{equation}
for all $x>0$ and
\begin{equation}\label{eq:K0upperboundM=0}
K_0(x) \ \le \ 	- \left (\ln \frac{x}{2} + \gamma \right ) +  0.79 \  \left ( \frac{x^2}{4} \right ) \abs{ \ln \frac{x^2}{4}}
\end{equation}
for $x < 2 \e^{-\gamma}$.  In \eqref{eq:K0upperboundM=0}, we have used the fact that $I_0$ is increasing so that
$$\frac{I_0(x) \e  }{4 \pi} \ \le \ \frac{I_0(2\e^{-\gamma}) \e^2  }{4\pi} \ = \ 0.7885 \cdots \ \le \ 0.79.$$
\end{remark}
\begin{proof} Let $M$ be fixed and let 
$$\Phi(x) = - \left (  \ln \frac{x }{2} + \gamma \right ) + \sum_{n=1}^{M} \frac{1}{n!^2} \left ( \frac{x^2}{4} \right )^n \left (  -\ln \frac{x }{2} - \gamma + h_n \right ).$$
Since $h_n$ increases with $n$, if 
$$x \ge  2 \e^{-\gamma} \e^{h_M}$$
then every term contributing to $\Phi(x)$ is non-positive, so $\Phi(x) \le 0$ and the bound $\Phi(x) < K_0(x)$ is trivial.  However, for $x < 2 \e^{-\gamma} \e^{h_M}$ 
\begin{equation}
K_0(x) - \Phi(x) \ = \ \sum_{n=M+1}^\infty \frac{1}{n!^2} \left ( \frac{x^2}{4} \right )^n \left (  -\ln \frac{x }{2} - \gamma + h_n \right ) \label{eq:K0-Phi}
\end{equation}
is a sum of positive terms and thus, $K_0(x) \ge \Phi(x)$.  Eq.\ \eqref{eq:K0lowerbound} follows.

To derive the upper bound \eqref{eq:K0upperbound}, fix $x < 2 \e^{-\gamma}$. Since 
$$(n+1)h_n \ = \ n h_n + h_n \ = \ n h_{n+1} + h_n - \frac{n}{n+1} \ \ge \ n h_{n+1},$$
it follows that $\nicefrac{h_n}{n}$ is decreasing with $n$.  Thus
$$
	K_0(x) - \Phi(x) \ \le  \ \left ( \frac{x^2}{4} \right )^{M+1} \sum_{n=M+1}^\infty  \frac{1}{n!(n-1)!} \left ( \frac{x^2}{4} \right )^{n-M-1} \left ( \frac{ - \ln \frac{x }{2} - \gamma}{n} + \frac{h_{M+1}}{M+1}  \right )
$$
Since  $ h_{M+1} \ \le \ \gamma + \ln (M+1) $, $(n+k)! \ge k! n!$ and $-\ln \nicefrac{x}{2} -\gamma > 0$, we conclude that
\begin{multline*} K_0(x) - \Phi(x) \ \le \  \frac{-\ln \frac{x}{2} +\ln (M+1)}{(M+1)!^2 } \left ( \frac{x^2}{4} \right )^{M+1} \sum_{n=0}^\infty \frac{1}{n!^2} \left( \frac{x^2}{4} \right )^{n} \\
\ = \ \frac{-\ln \frac{x}{2} +\ln (M+1)}{(M+1)!^2 } \left ( \frac{x^2}{4} \right )^{M+1} I_{0}(x).
\end{multline*}

By Stirling's approximation
$$ n!  \ \ge \ \sqrt{2\pi} n^{n+\frac{1}{2}} \e^{-n}$$
we see that
$$\frac{1}{(M+1)!^2} \left ( \frac{x^2}{4} \right )^{M+1} \ \le  \ \frac{e^{2M+2}}{2\pi(M+1)} \left ( \frac{x}{2(M+1)} \right )^{2M+2}.$$
Eq.\ \eqref{eq:K0upperbound} follows.

In the range $2 \e^{-\gamma}\le x < 2\e^{h_M -\gamma}$, we have $-\ln \frac{x}{2} - \gamma < 0$ so 
\begin{multline*}K_0(x) - \Phi(x) \ \le \ 
	\left ( \frac{x^2}{4} \right )^{M+1} \sum_{n=M+1}^\infty  \frac{1}{n!(n-1)!} \left ( \frac{x^2}{4} \right )^{n-M-1} \ \frac{h_{M+1}}{M+1} \\
	\le \ \frac{\ln (M+1)}{(M+1)!^2 } \left ( \frac{x^2}{4} \right )^{M+1} I_{0}(x)\ \le \ \frac{I_0(x)}{2\pi}\frac{e^{2M+2}(\gamma+\ln (M+1))}{M+1} \left ( \frac{x}{2(M+1)} \right )^{2M+2},
\end{multline*}
from which eq.\ \eqref{eq:K0upperboundbigx} follows.
\end{proof}


\begin{thebibliography}{9}
\providecommand{\natexlab}[1]{#1}
\providecommand{\url}[1]{\texttt{#1}}
\providecommand{\urlprefix}{URL }
\expandafter\ifx\csname urlstyle\endcsname\relax
  \providecommand{\doi}[1]{doi:\discretionary{}{}{}#1}\else
  \providecommand{\doi}{doi:\discretionary{}{}{}\begingroup
  \urlstyle{rm}\Url}\fi
\providecommand{\eprint}[2][]{\url{#2}}

\bibitem[{Adams et~al.(2017)Adams, Schenker, McGhee, Gut, Brunner and
  Miller}]{TrappingPaper}
C.~G. Adams, J.~H. Schenker, P.~S. McGhee, L.~J. Gut, J.~Brunner and J.~R.
  Miller.
\newblock Maximizing information yield from pheromone-baited monitoring traps:
  Estimating plume reach, trapping radius, and absolute density of codling moth
  ({\it cydia pomonella}) in michigan apple.
\newblock \emph{{\it J. Econ. Ent.}} \textbf{110}, 305--318 (2017).

\bibitem[{Carslaw and Jaeger(1940)}]{Carslaw1940a}
H~S Carslaw and J~C Jaeger.
\newblock {Some Two-Dimensional Problems in Conduction of Heat with Circular
  Symmetry}.
\newblock \emph{Proc. London Math. Soc.} \textbf{s2-46}~(1), 361--388 (1940).
\newblock \href{http://www.ams.org/mathscinet-getitem?mr=0002454}{0002454}.

\bibitem[{Collet et~al.(2000)Collet, Mart{\'i}nez and
  San~Mart{\'i}n}]{Collet2000}
P. Collet, S. Mart{\'i}nez and J. San~Mart{\'i}n.
\newblock Asymptotic behaviour of a brownian motion on exterior domains.
\newblock \emph{Prob. Th. Rel. Fields} \textbf{116}~(3),
  303--316 (2000).
\newblock \href{http://www.ams.org/mathscinet-getitem?mr=1749277}{1749277}.

\bibitem[{{\relax DLMF}()}]{DLMF}
{\relax DLMF}.
\newblock {\it NIST Digital Library of Mathematical Functions}.
\newblock http://dlmf.nist.gov/, Release 1.0.13 of 2016-09-16 (2016).
\newblock F.~W.~J. Olver, A.~B. {Olde Daalhuis}, D.~W. Lozier, B.~I. Schneider,
  R.~F. Boisvert, C.~W. Clark, B.~R. Miller and B.~V. Saunders, eds.

\bibitem[{Hunt(1956)}]{Hunt1956}
G~A Hunt.
\newblock {Some theorems concerning Brownian motion}.
\newblock \emph{Trans. Am. Math. Soc.} \textbf{81}~(2), 294 (1956).
\newblock \href{http://www.ams.org/mathscinet-getitem?mr=0079377}{0079377}.

\bibitem[{Miller et~al.(2015)Miller, Adams, Weston and Schenker}]{Miller2015}
J.~R. Miller, C.~G. Adams, P.~A. Weston and J.~H. Schenker.
\newblock \emph{{Trapping of Small Organisms Moving Randomly}}.
\newblock SpringerBriefs in Ecology. Springer International Publishing, Cham
  (2015).

\bibitem[{M{\"o}rters and Peres(2010)}]{Brownian}
P.~M{\"o}rters and Y.~Peres.
\newblock \emph{Brownian Motion}.
\newblock Cambridge Series in Statistical and Probabilistic Mathematics.
  Cambridge University Press (2010).
\newblock \href{http://www.ams.org/mathscinet-getitem?mr=2604525}{2604525}.

\bibitem[{Nicholson(1921)}]{Nicholson1921}
J.~W. Nicholson.
\newblock {A Problem in the Theory of Heat Conduction}.
\newblock \emph{Proc. R. Soc. A Math. Phys. Eng. Sci.} \textbf{100}~(704),
  226--240 (1921).

\bibitem[{Wendel(1980)}]{Wendel1980}
J~G Wendel.
\newblock {Hitting Spheres with Brownian Motion}.
\newblock \emph{Ann. Probab.} \textbf{8}~(1), 164--169 (1980).
\newblock \href{http://www.ams.org/mathscinet-getitem?mr=556423}{556423}.

\end{thebibliography}
\end{document}